\documentclass[11pt]{amsart}
\usepackage{xparse,amssymb,amsmath,amsthm,mathrsfs}

\usepackage{tabularx,array}
\usepackage[top=1in, bottom=1.25in, left=1.25in, right=1.25in]{geometry}
\usepackage{amssymb,amsmath,amsthm,mathrsfs}
\usepackage{enumitem}
\usepackage{tikz-cd}
\usetikzlibrary{cd}
\usepackage{wrapfig}

\usepackage[normalem]{ulem}
\usepackage{url}
\usepackage{cite}
\usepackage{amsfonts}
\usepackage{verbatim}
\usepackage{graphicx}
\usepackage{caption} 
\usepackage[all]{xy}
\usepackage{colonequals}
\usepackage{subcaption}

\usepackage[utf8]{inputenc}
\usepackage{geometry}
\usepackage{amsfonts}
\usepackage{graphicx}
\usepackage{mathrsfs}
\usepackage{amsmath}
\usepackage{mathtools}
\usepackage{amssymb,amsmath,amsthm,mathrsfs}
\usepackage{enumitem}
\usepackage{float}
\usepackage{xcolor}
\usepackage{tikz}
\usepackage{fancyvrb}
\usepackage{lipsum}
\usepackage{lmodern}
\usepackage{tcolorbox}
\usepackage{float}
\usepackage[title]{appendix}

\numberwithin{equation}{subsection}
\theoremstyle{plain}
\newtheoremstyle{case}{}{}{}{}{}{:}{ }{}

\newtheorem{thm}{Theorem}[section]
\newtheorem*{thm*}{Theorem}

\newtheorem{prop}[thm]{Proposition}
\newtheorem{cor}[thm]{Corollary}
\newtheorem{lem}[thm]{Lemma}
\newtheorem{remark}[thm]{Remark}

\def\C{{\mathbb{C}}}
\def\Z{{\mathbb{Z}}}

\def\R{{\mathbb{R}}}

\def\Q{{\mathbb{Q}}}

\def\lcm{{\text{\rm{lcm}}}}

\def\lmd{\,\middle|\,}

\newcommand\prt[1]{\left(#1\right)}

\newcommand\SL[1]{{\text{\rm{SL}}}\left(#1\right)}

\newcommand\lset[1]{\left\{#1\right\}}

\newcommand\labs[1]{\left|#1\right|}
\newcommand\lbrkt[1]{\left[#1\right]}

\usepackage{thmtools}
\usepackage{thm-restate}

\usepackage{hyperref}

\usepackage{cleveref}

\title{Translation Surfaces arising from Right Regular Prisms}
\author{Xun Gong}
\address{Department of Mathematics, The Ohio State University\\ 100 Math Tower, 231 W 18th Ave, Columbus, OH 43210, USA}
\email{gong.691@buckeyemail.osu.edu}

\author{Zuo Lin}
\address{Department of Mathematics, University of California, Berkeley\\ 970 Evans Hall, Berkeley, CA 94720, USA}
\email{zuo\_lin@berkeley.edu}

\author{Anthony Sanchez}
\address{Department of Mathematics, University of Houston\\ Philip Guthrie Hoffman Hall
3551 Cullen Blvd, Room 641
Houston, Texas 77204, USA}
\email{asanchez109@uh.edu}
\begin{document}

\maketitle

%\tableofcontents
\begin{abstract}
%We consider a right regular $n$-prism as an $n$-differential. We find the unfolding and the primitive surface the unfolding covers. We show the base is hyperelliptic and not a lattice surface. Combining these properties with the classification of Apisa \cite{MR3769676} of hyperelliptic components of strata we compute the  $\text{GL}(2,\mathbb R)$ orbit closure of the base. The orbit closure is used to get exact asymptotics on the base surface, the unfolding, and the right regular $n$-prism.

We study flat metrics arising from right regular $n$-prisms by viewing them as $n$-differentials and analyzing their associated unfoldings. We show that the unfolding of a right regular $n$-prism is never a lattice surface unless $n=4$, in contrast with the case of Platonic solids. Despite this, we prove that these surfaces admit translation coverings to hyperelliptic surfaces, allowing us to determine their $\mathrm{GL}(2,\mathbb{R})$-orbit closures using the classification of hyperelliptic components of strata.

As a consequence, we obtain exact quadratic asymptotics for a certain average of the number of saddle connections on the base surfaces, their unfoldings, and the original prisms, including their Siegel–Veech constants. This provides a natural infinite family of non-lattice surfaces for which orbit closures and counting problems can be computed explicitly.
\end{abstract}

\section{Introduction}
Unfoldings of polyhedral surfaces provide a rich source of translation surfaces whose geometry and dynamics reflect the underlying geometric structures. While unfoldings of Platonic solids give rise to lattice surfaces (Athreya--Aulicino--Hooper \cite{MR4477409}), much less is understood for other families. In this paper, we study \emph{right regular $n$-prisms}, which form a natural infinite family where the lattice property fails for the unfolding, but sufficient structure remains to analyze orbit closures and counting problems.

More precisely, we consider the class of polyhedra given by right regular $n$-prisms of height 1 that we denote by $P_n$ for $n\in\Z_{\geq 3}$. The flat metric on the surface of $P_n$ is given by an $n$-differential. We always orient the net, also denoted by $P_n$, so that a bottom edge is horizontal. Right regular $n$-prisms, thought of as singular flat 3-manifolds, were also considered in Athreya--B\'edaride--Hooper--Hubert \cite{Athreyaprisms}, where ergodicity of linear flow was studied.

We will associate to each $n$-prism a geometric object called a \emph{translation surface} through the process of \emph{unfolding}. In short, a translation surface is a polygon in the plane with edge identifications given by translations, and unfolding is a procedure that allows us to turn trajectories on $P_n$ into straight lines on a surface. Interesting information about the underlying $n$-differential can be discovered in the associated translation surface, where more geometric, topological, and dynamical structures are present. 

 As an example, the unfolding reveals information about the saddle connections on the $n$-differential. By a saddle connection, we mean geodesic paths with respect to the flat metric, between cone points with no other cone points on the interior. See Figure~\ref{TildeP5} for an example of a saddle connection on the  right regular pentagonal prism. Saddle connections can be understood by considering the unfolding of the $n$-differential on a right regular prism and understanding the orbit closure in the moduli space of translation surfaces. This information allows one to count the number of saddle connections on the $n$-differential by reducing to counting saddle connections on the unfolding. More generally, given a $k$-differential $X$, we consider the counting function
 $$N(X,L)=\#\{\text{saddle connections on $X$ of length  }\le L\}.$$

For translation surfaces, which correspond to $k=1$, the influential work of Eskin--Masur \cite{MR1827113} proved that almost every translation surface has exact quadratic asymptotics. That is,
$$\lim_{L\to\infty}\frac{N(X,L)}{\pi L^2}$$
 converges to a constant called the  \emph{Siegel-Veech constant}
(re-normalized by the area of the surface). Later, the breakthrough work of Eskin--Mirzakhani--Mohammadi \cite{MR3418528} proved an everywhere result for an average over the counting function. That is, for every translation surface $X$, we have
\begin{equation}\label{eqn: EMM}
\lim_{L\to\infty}\frac{1}{L}\int_0^LN(X,e^t)e^{-2t}dt
\end{equation}
exists and equals the Siegel-Veech constant
(re-normalized by the area of the surface) associated to the minimal orbit closure containing $X$.

By utilizing the unfolding of $n$-differentials on right regular prisms and the deep literature on translation surfaces, we can compute the Siegel-Veech constant. See also Aygun \cite{Kdiff} and the references therein where weak asymptotics of saddle connections for some $k$-differentials on higher genus surfaces are studied. For other examples of the unfolding revealing information about the underlying differentials, see, Athreya--Aulicino--Hooper \cite{MR4477409}, Athreya--Lee \cite{MR4238630}, Athreya--Aulicino \cite{MR3910631},  Gong--Sanchez \cite{MR4689114}, and Guti\'errez-Romo--Lee--Sanchez \cite{GLS}.

Using  how  the unfolding is situated in the moduli space of translation surfaces as a way  to gain information about the underlying differential is also present in the work of Athreya--Aulicino--Hooper \cite{MR4477409} where their interest was in the existence of closed saddle connections on the Platonic solids. 

Our work utilizes the unfolding for an infinite family of natural geometric bodies given by the right regular prisms. In contrast to Athreya--Aulicino--Hooper \cite{MR4477409} where their differentials were tiled by the same regular polygon, our surfaces are tiled by two regular polygons. Notably, this seemingly small difference results in the unfolding of a right regular prism never being a lattice surface (unless $n=4)$ which is in sharp contrast to the case of platonic solids where the unfoldings are always lattice surfaces. Despite the fact that the unfoldings are not lattice surfaces, we show that there is sufficient structure such as hyperellipticity in a translation covering of the unfolding that we can appeal to the classification of Apisa \cite{MR3769676} of hyperelliptic components of strata to compute the orbit closure.

Let us denote the unfolding of $P_n$ by $\tilde P_n$. The following theorem that computes the combinatorial cone data of $\tilde P_n$ and shows there is a translation covering to a lower genus surface was first discovered by David Aulicino, Alisa Leshchenko, and Eric Loucks. See Section \ref{Preliminaries} for formal definitions of these notions and Figure~\ref{TildeP5} for examples of the surfaces that appear in this article. 

\begin{thm}\label{thm:UnfoldingsResult}
    The surface $\tilde{P}_n$ has the genus $(n-1)^2$ and lies in the stratum $\mathcal H\left((n-2)^{2n}\right)$. The surfaces $\tilde{P}_n$ are never primitive. 
\end{thm}

Note that we are using the shorthand exponential notation for strata so that $\mathcal H\left((n-2)^{2n}\right)$ means there are $2n$ cone points of order $n-2$.

We continue the study of right regular prisms and show that the unfoldings are not lattice surfaces unless $n=4$. This is in contrast to Athreya--Aulicino--Hooper \cite{MR4477409} where the unfoldings of the platonic solids are always lattice surfaces.

\begin{thm}\label{thm:UnfoldingsVeech}
 The unfoldings $\tilde{P}_n$ are lattice surfaces if and only if $n=4.$ In this case, $\tilde{P}_4$ is a square-tiled surface.
\end{thm}

Both of these results follow by constructing translation surfaces $\Pi_n$ that have the same underlying polygon as $P_n$, but with different identifications. Depending on the parity of $n$, we can find the primitive cover and prove hyperellipticity which allows us to use the classification of Apisa \cite{MR3769676} of hyperelliptic components of strata to compute the orbit closure. This provides a natural example of a non-lattice family of translation surfaces for which orbit closures and Siegel–Veech asymptotics can still be computed explicitly.

\begin{thm} \label{thm:main}

    \begin{enumerate}
        \item If $n\ge3$ is even and $n\ne4$, there is a degree $2n$ translation cover $\tilde P_n\to \Pi_n ^h$ to a non-lattice primitive hyperelliptic surface $\Pi_n ^h\in \mathcal H(n-2)$. Furthermore, the $\text{GL}(2,\mathbb R)$ orbit closure of $\Pi_n ^h$ is exactly the set of hyperelliptic translation surfaces in $\mathcal H(n-2)$. As such, we have asymptotics for the number of saddle connections in a ball. That is, there is a constant $c=c(\mathcal H(n-2))>0$ such that 
        $$\lim_{L\to\infty}\frac{1}{\pi^2L}\int_0^LN(\Pi^h _n,e^t)e^{-2t}dt=\frac{c}{\text{Area}(\Pi_n^h)}.$$
             %%%%%%%%%%%%%%%%%%%%%%%%%%%%%%%%%%%%%%%%%%%%%%%%%%%%%%%%
        \item If $n\ge3$ is odd, there is a degree $n$ translation cover $\tilde P_n\to \Pi_n$ to a non-lattice primitive hyperelliptic surface $\Pi_n\in \mathcal H(n-2,n-2)$. Furthermore, the $\text{GL}(2,\mathbb R)$ orbit closure of $\Pi_n$ is exactly the set of hyperelliptic translation surfaces in $\mathcal H(n-2,n-2)$. As such, we have asymptotics for the number of saddle connections in a ball. That is, there is a constant $d=d(\mathcal H(n-2,n-2))>0$ such that 
        $$\lim_{L\to\infty}\frac{1}{\pi^2L}\int_0^LN(\Pi_n,e^t)e^{-2t}dt=\frac{d}{\text{Area}(\Pi_n)}.$$
    \end{enumerate}
\end{thm}

The constant $c$ (resp. $d$) appearing above coincides with the Siegel-Veech constant appearing in Eskin--Masur--Zorich \cite{MR2010740} for the stratum $\mathcal H(n-2)$ (resp. $\mathcal H(n-2,n-2)$). Asymptotics when $n=4$ is work of Athreya--Lee \cite{MR4238630} as $\tilde P_4$ corresponds to the platonic solid given by the cube.

As a corollary of the above and using properties specific to the covers in our situation, we are able to show the following.

\begin{cor}\label{cor:asymforPn}
    For $n\ge3$ even and $n\ne4$, we have
            $$\lim_{L\to\infty}\frac{1}{\pi^2L}\int_0^LN(P_n,e^t)e^{-2t}dt=\frac{4c}{\text{Area}(P_n)}$$
            where $c$ is the same constant appearing in Theorem \ref{thm:main}.
            
    For $n\ge3$ and odd, we have
            $$\lim_{L\to\infty}\frac{1}{\pi^2L}\int_0^LN(P_n,e^t)e^{-2t}dt=\frac{d}{\text{Area}(P_n)}$$
             where $d$ is the same constant appearing in Theorem \ref{thm:main}.

\end{cor}

\begin{remark}
     In Athreya--Aulicino \cite{MR3910631}, they showed a closed saddle connection on the dodecahedron exists. Such a path does not exist on any other platonic solids see e.g. Athreya--Aulicino--Hooper \cite{MR4477409}. We remark that for $n\ge5$ and odd, there exists a closed saddle connection by considering the path going from bottom-most point of the net to the top-most. See Figure~\ref{TildeP5} for an example of this when $n=5$. It would be interesting to see if such a path exists on a right regular prism for even $n\ge5$.
\end{remark}

\begin{comment}
    \begin{remark}
Our techniques and results extend to the infinite family of regular \emph{antiprisms}. For brevity, our proofs are only written for prisms, but in the final section we make a remark indicating the modifications needed in our proofs.
\end{remark}
\end{comment}
  \begin{figure}
    \centering
    \includegraphics[scale=.65]{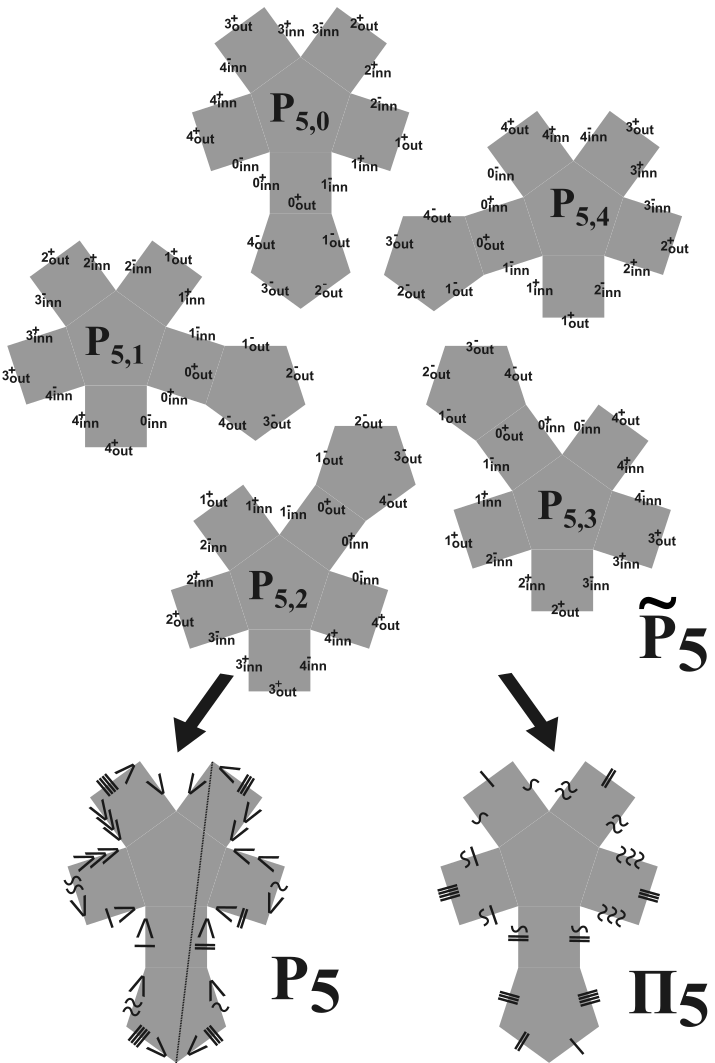}
    \caption{On the bottom left is the right regular 5-prism $P_5$ along with a closed saddle connection. At the top is the unfolding of $P_5$ which we denote as $\tilde P_5$. The identifications of $\tilde P_5$ can be found in Equation~\ref{eqn:identificationsof TildePn}. On the bottom right is $\Pi_5$ which, by Proposition \ref{CoveringMap}, has $\tilde P_5$ as a 5-fold cover.  } \label{TildeP5}
\end{figure}

\textbf{Acknowledgments.}  The authors would like to thank Paul Apisa, Jayadev Athreya, David Aulicino, and Pat Hooper for helpful conversations. Additionally, the authors want to thank Alex Rhee and Yi Yao who were part of a directed reading at UCSD where this project began.
At the start of this project, A.S. was supported by the National Science Foundation Postdoctoral Fellowship under grant number DMS-2103136.
%%%%%%%%%%%%%%%%%%%%%%%%%%%%%%%%%%%%%%%%%%%%%%%%%%%%%%%%%%%%%%%%%%%%%%%%%%%%%%%%%%%%%%%
\section{Preliminaries}\label{Preliminaries}

We briefly review the basic theory of $k$-differentials and translation surfaces to aid the reader. We refer to Athreya--Masur \cite{MR4783430} for more details. 

Let $k$ be a positive integer and $X$ be a Riemann surface. A meromorphic $k$-\emph{differential} on $X$ is a differential $\xi$ on $X$ which, in local coordinates, can be expressed in the form $\xi=f(z)(dz)^k$ where $f$ is meromorphic. A zero or pole of order $j>-k$ corresponds to a \emph{cone point} of angle $\frac{(j+k)2\pi}{k}.$ 

By the Riemann-Roch theorem, the sum of the orders of zeros and poles of $\xi$ of a $k$-differential $(X,\xi)$ is $k(2g-2)$ where $g$ denotes the genus of $X$. The space of genus $g$ $k$-differentials can be stratified by a partition $(m_1,\dots,m_s)$ of $k(2g-2)$ where $m_j>-k$ for all $j$. The moduli space of $k$-differentials with multiplicities of the zeros and poles given by $(m_1,\ldots,m_s)$ is denoted as $\mathcal H_k(m_1,\ldots,m_s)$.

When $k=1$, a $k$-differential is called a \emph{translation surface}. That is, a translation surface is a pair $(X,\omega)$ where $\omega$ is a holomorphic non-zero 1-form on $X.$ Occasionally, we suppress the one-form $\omega$ and denote the translation surface simply as $X$.  Translation surfaces can be defined geometrically as finite collections of polygons with sides identified by a translation. The total angle about each vertex of the polygon presentation of a translation surface is necessarily an integer multiple of $2\pi$. The vertices where the total angle is greater than $2\pi$ are called cone points and correspond to the zeros of the $1$-form. In fact, a zero of order $m$ is a cone point of total angle $2\pi(m+ 1)$. We denote the set of cone points of a translation surface $X$ by $\Sigma(X)$. Given an integer partition of $2g-2$, we denote the moduli space of translation surfaces with multiplicities of the zero given by $(m_1,\ldots,m_s)$  as $\mathcal H(m_1,\ldots,m_s)$ and remove the subscript $\mathcal H_1(\cdot)$ in this case.

Associated to a $k$-differential is a translation surface called the \emph{unfolding} or \emph{holonomy cover}. This surface is formed by unfolding the $k$-differential.

A translation surface inherits a flat metric from $\C$. A  \emph{saddle connection} on a translation surface is a straight line geodesic connecting two cone points with no cone points in the interior. A \emph{cylinder} on a translation surface is an isometrically embedded copy of a Euclidean cylinder $(\R/c\Z)\times(0, h)$ whose boundary is a union of saddle connections. The number $c$ is the \emph{circumference}, and the number $h$ is the \emph{height} of the cylinder. The ratio $h/c$ is the \emph{modulus} of the cylinder. 

Given a $k$-differential $(X,\xi)$, there is a notion of cylinder and saddle connections. A cylinder is the union of closed geodesics disjoint from singularities in the same isotopy class rel singularities. A geodesic between two singularities with no singularities in the interior is called a saddle connection. When $k>2$, cylinders and saddle connections may self-intersect, which complicates their study.

There is a natural $\text{SL}(2,\mathbb{R})$ action on translation surfaces coming from the linear action of matrices on $\mathbb{R}^2$. The \emph{Veech group} of a translation is the stabilizer of $(X,\omega)$ under the $\SL{2,\R}$-action: \begin{align*}
        \SL{X,\omega}=\lset{g\in\SL{2,\R}\lmd g(X,\omega)=(X,\omega)}.
    \end{align*} 
     We call a surface a \emph{lattice surface} if its Veech group is a lattice in $\SL{2,\R}$.

    Let $(X,\omega)$ and $(Y,\eta)$ be translation surfaces. A surjective holomorphic  map $p:(X,\omega)\to (Y,\eta)$ is a \emph{translation cover} if the pullback form $p^*(\omega)=\eta$. We call $(Y,\eta)$ the \emph{base surface} and $(X,\omega)$ the \emph{covering surface} of $p$. We call $(X,\omega)$ \emph{primitive} if it does not admit a non-trivial translation covering to a lower genus translation surface. 

    The existence of a primitive translation surface is guaranteed due to the following result of M\"oller.

    \begin{thm}\label{thm:Moller}(M\"oller \cite{MR2242629})
        Every translation surface covers a primitive translation surface. If the genus of the primitive surface is greater than 1, then this covering is unique.
    \end{thm}

    The following result of Gutkin--Judge \cite{GJ} connects the Veech group of translation surfaces that differ by a translation covering.

\begin{thm}\label{Thm:TranslationCovering} (Gutkin--Judge \cite{GJ})
    If $p:(X,\omega)\to (Y,\eta)$ is a translation covering, then their Veech groups are commensurable. That is, there is $g\in\mathrm{SL}(2,\mathbb R)$ such that $\mathrm{SL}(X,\omega)\cap g\mathrm{SL}(Y,\eta)g^{-1}$ has finite index in both $\mathrm{SL}(X,\omega)$ and $g\mathrm{SL}(Y,\eta)g^{-1}$.
\end{thm}

We close with the Veech dichotomy and a corollary that puts strong restrictions on the moduli of cylinders for lattice surfaces.
    
\begin{thm}[Veech Dichotomy \cite{MR1005006}]\label{VeechDichotomy}
    If the Veech group $\SL{X,\omega}$ is a lattice in $\SL{2,\R}$, then for any direction, either the foliation in that direction is minimal, or every leaf is closed or a saddle connection.
\end{thm}
\begin{cor}\cite[Lemma 3.8]{Ya}\label{CorVeech}
    Suppose that the Veech group $\SL{X,\omega}$ is a lattice in $\SL{2,\R}$. Then all horizontal cylinders have commensurable moduli, i.e., their ratio of height and circumference differs by a rational multiple.
\end{cor}

%%%%%%%%%%%%%%%%%%%%%%%%%%%%%%%%%%%%%%%%%%%%%%%%%%%%%%%%

\section{Unfoldings, covers and hyperellipticity}

In this section, we consider the unfoldings of flat metric on the surface of a right regular $n$-prism, show they are never primitive, and that they cover a hyperelliptic translation surface.

 By Lemma 2.5 of Athreya--Aulicino--Hooper \cite{MR4477409}, we immediately obtain the stratum of the unfolding by noting that each cone point is a simple pole. We summarize this below.

\begin{lem}\label{lem: unfoldings tilde Pn}
    The flat metric on the surface a right regular $n$-prism is an $n$-differential with no zeros and $2n$ simple poles. Thus, we have $P_n\in\mathcal H_n(-1^{2n})$ and of genus 0 and the canonical $n$-cover $\tilde P_n$ lies in the stratum $\mathcal H((n-2)^{2n})$ and has genus $(n-1)^2$.
\end{lem}
While Lemma 2.5 provides us with the cone data and start of $\tilde P_n$, it does not make clear the side identifications. As such, we provide the unfolding procedure of Zemlyakov--Katok \cite{MR4893865}.

 The polygon of the unfolding of $P_n$ consists of $n$ copies of the net (without identifications) where each copy is rotated by $2\pi/n$. We refer to the ``top" $n$-gon as the \emph{center}, the square panels emanating from the center as the \emph{petals}, and the ``bottom"  $n$-gon as the \emph{base}. Often we refer to the petal connecting the center and base as the \emph{stem}. We establish the identifications shortly after setting up some notation of $\tilde P_n$. See Figure \ref{TildeP5} for the case $n=5$ as a reference.

 We label each copy of the $P_n$ starting from the non-rotated copy of $P_n$ labeled as $P_{n,0}$ and label the other rotated copies in counterclockwise order as $P_{n,1}, \dots, P_{n,n-1}$. For convenience, we also label each rotated copy of $P_n$ at the edge that connects the stem and the base. Furthermore, we call the edges on the petals that are perpendicular to the center the \emph{inner edges}, and all other edges on the petal not attached to the center the \emph{outer edges}. We label edges on every $P_{n,k}$ in the following way:
    \begin{itemize}
        \item Inner edges: any edge $i$ that was previously identified in the net $P_n$ will be labeled as $i_{inn}^+$ and $i_{inn}^-$ with $i_{inn}^+$ oriented rightward $i_{inn}^-$. We start by labeling the inner edge of the stem with $0_{inn}^+$ on the left and $1_{inn}^-$ on the right. The labels $$\lset{0_{inn}^-, 0_{inn}^+,\dots,(n-1)_{inn}^-,(n-1)_{inn}^+}$$ are ordered counterclockwise.
        \item Outer edges: the edge in each $P_{n,i}$ between the stem and the base will be labeled $0_{out}^+$ on the stem and $0_{out}^-$ on the base. We label the outer edges of the petals $\lset{0_{out}^+,\dots,(n-1)_{out}^+}$ ordered counterclockwise and of the base $\lset{0_{out}^-,\dots,(n-1)_{out}^-}$ ordered clockwise.
    \end{itemize} Thus, each edge in $\tilde{P}_n$ has a unique label $\prt{i_{inn, out}^{+,-},k}$, for each $ i=1,\ldots, n-1,$ and $k\in\Z/n\Z$  where $k$ represents that the edge is on the $k$-th rotated copy of $P_n$. 
    
    The unfolding procedure indicates how the edge identifications in $\tilde{P}_n$ are for each $i, k\in\Z/n\Z$: \begin{equation}\label{eqn:identificationsof TildePn}
    \prt{i_{inn}^+,k}\sim\prt{i_{inn}^-,k+1}\ \text{and}\ \prt{i_{out}^+,k}\sim\prt{i_{out}^-,k+2i}.
    \end{equation}
    
    In particular, by pasting the base of each rotated  $n$-gon $\prt{i_{out}^-,i}$ to the square edge $\prt{i_{out}^+,-i}$,  we obtain $n$ translated copies of $P_{n,0}$. That is, $n$-copies of the same polygon as the net.
    
    Let $\Pi_n$ denote one such translated copy. See Figure~\ref{Pi(6,h)} for an example when $n=6.$ This defines a translation surface with the same net as $P_n$, but with different identifications.  The edge identifications for $\Pi_n$, which are implicit in Equation~\ref{eqn:identificationsof TildePn}, are given by the following:
\begin{itemize}
    \item[(C-1)] We identify the two edges of a petal that are perpendicular to the center and base edges.
    \item[(C-2)] For the remaining edge $s$ on a stem, we note that the opposite side of the square that it is attached to on the center has a unique edge $s'$ on the base $n$-gon that is a $180^\circ$ rotation of the center $n$-gon. We attach $s$ to $s'$ by translation.
\end{itemize}

    This discussion has proved the following:

\begin{figure}
    \centering
    \includegraphics[scale=.6]{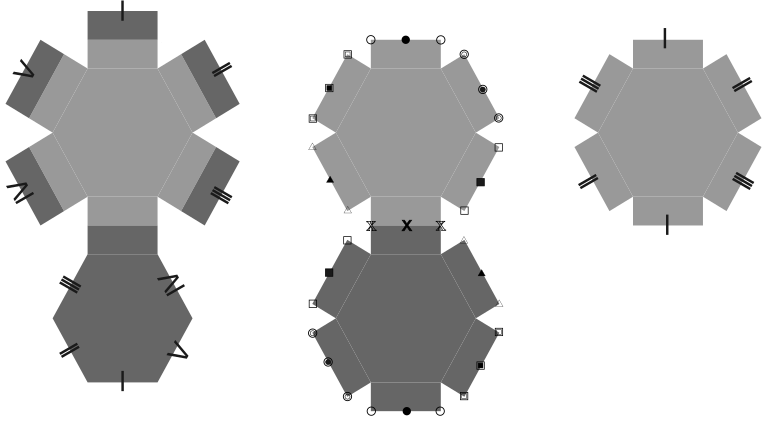}
    \caption{On the left is $\Pi_6$ with identifications for the outer edges. Inner edges on the same petal are identified. In the middle is a cut and paste of $\Pi_6$ along with the fixed points of the hyperelliptic involution. On the right is $\Pi_6 ^h$ with identifications of the outer edges.} \label{Pi(6,h)}
\end{figure}

\begin{prop}\label{CoveringMap}
    The surface $\tilde P_n$ is never primitive. In fact, there is a degree $n$ translation covering  $p:\tilde P_n\to \Pi_n$.
\end{prop}

We record the genus and combinatorial data of the cone points of $\Pi_n$ in the following theorem. 

\begin{thm}\label{Pi_n genus and stratum}
    The surface $\Pi_n$ has genus $n-1$ and lies in the stratum $\mathcal{H}\prt{n-2,n-2}$.
\end{thm}
\begin{proof}
     Let $F(\cdot), E(\cdot), V(\cdot)$, denote the number of faces, edges, and vertices of a polygon. Each Euclidean polygon in $\Pi_n$ yields one face, so $F(\Pi_n)=n+2$. Moreover, since there are $4n$ edges from the petals and $2n$ edges from the base and center, with the edge identifications of $\Pi_n$, we have $E(\Pi_n)=\frac{4n+2n}{2}=3n$. Now, to count the number of vertices, we use the same notion as of Proposition \ref{CoveringMap} to divide the vertices of $\Pi_n$ into inner points and outer points. The inner points contribute one vertex count since it only depends on the identification of edges perpendicular to the attached edge of the square panels and the center. The outer points also contribute one vertex count since one alternatively identifies the petal vertices with the base vertices, and two outer points within each petal are identified as the same one. Hence, we have 
    \begin{align*}
        2-2g&=V(\Pi_n)-E(\Pi_n)+F(\Pi_n)\\
        &=2-3n+n+2\\
        &=4-2n,
    \end{align*} which implies $g=n-1$. Furthermore, by a cut and paste that move the petals from the center onto the base, inner points transform into outer points, so the cone angles of two vertices of $\Pi_n$ are the same. Since $n$ petals contribute $n\pi$ angle and the interior of the $n$-gon contributes $(n-2)\pi$, the total cone angle is $$n\pi+(n-2)\pi=(n-1)2\pi=((n-2)+1)2\pi,$$ which show that $\Pi_n$ lies in the stratum $\mathcal{H}\prt{n-2,n-2}$. \end{proof}

    When $n$ is even, we can demonstrate that $\Pi_n$ is not primitive. That is, there is a surface, which we denote by $\Pi_n ^h$ for which there is a degree two translation cover $\Pi_n
    \to\Pi_n ^h$. The surface $\Pi_n ^h$ is constructed by removing the base and keeping only half of each petal emanating from the center. See Figure~\ref{Pi(6,h)} for an example when $n=6$ and how this forms a degree two cover. This cover only holds when $n$ is even due to the fact that the base is invariant under 180-degree rotation.

In the following, we record the genus and combinatorial data of the cone points.

\begin{thm}
    The surface $\Pi_n^h$ has genus $\frac{n}{2}$ and lies in the stratum $\mathcal{H}\prt{n-2}$.
\end{thm}
\begin{proof}
    Let $F(\cdot), E(\cdot), V(\cdot)$ denote the number of faces, edges, and vertices of a polygon. Each Euclidean polygon in $\Pi_n^h$ yields one face, so $F\prt{\Pi_n^h}=n+1$. Moreover, since there are $4n$ sides from the half-petals and $n$ sides from the center, with edge identifications of $\Pi_n^h$, we have $E\prt{\Pi_n^h}=\frac{5n}{2}$. For the vertices, the same idea as Theorem \ref{Pi_n genus and stratum} provides that the inner points contribute one vertex count and outer points contribute $\frac{n}{2}$ vertex count. Hence, we have \begin{align*}
        2-2g&=V\prt{\Pi_n^h}-E\prt{\Pi_n^h}+F\prt{\Pi_n^h}\\
        &=\prt{\frac{n}{2}+1}-\frac{5n}{2}+\prt{n+1},
    \end{align*} which shows $g=\frac{n}{2}$. The cone angle of the vertices from the outer points is $2\pi$ since they only depend on two corners in each of the opposite two half-petals. For the cone angle of the vertex from the inner points, it depends on the interior angle of $n$-gon and two corners of each half-petal. Therefore, we have 
        $$n\pi+(n-2)\pi=(n-1)2\pi=((n-2)+1)2\pi,$$
     which shows that $\Pi_n ^h$ lies in the stratum $\mathcal{H}\prt{n-2}$.
\end{proof}

We end this section by proving that the surfaces $\Pi_n$ are hyperelliptic. 

A $g$ translation surface $(X,\omega)$ is \emph{hyperelliptic} if there exists an involution $\sigma:X\to X$ such that $\sigma^*(\omega)=-\omega$. We call $\sigma$ the hyperelliptic involution. Geometrically, with respect to the flat structure of the translation surface, $\sigma$ corresponds to a 180-degree rotation. In this case, the map $\sigma$ has $2g+2$ fixed points.

\begin{prop}\label{prop:hyperelliptic}
    The surface $\Pi_n$ is hyperelliptic.
\end{prop}
\begin{proof}
    We first do a cut-and-paste on $\Pi_n$ by cutting half of the petals and gluing them to the base as in the middle photo in Figure \ref{Pi(6,h)}. 
    
    Let $\sigma_n: \Pi_n\to \Pi_n$ be the map that rotates our surface by angle $\pi$ around the center of the stem, which is conformal and acts as an involution. Notice that since $\sigma_n$ rotates every point in $\Pi_n$ by $\pi$ but preserves the edge identifications, any fixed points can only occur at the edges. As in the proof of Theorem \ref{Pi_n genus and stratum}, the vertices of the center and the base are different, so applying $\sigma_n$ still distinguishes between those vertices.  Furthermore, the edge identifications within each half-petal imply that the fixed points here can occur only at the outermost half-petal edges and the stem edges connecting the center and base. Thus, the fixed points are the outermost half-petal vertices, the midpoint of the outermost half-petal/connecting stem edges, and the center of rotation (the center of the stem). Thus, the number of fixed points of $\sigma_n$, is equal to $(n-1)+n+1=2n$. See Figure~\ref{Pi(6,h)} for an example when $n=6$.
    
    Theorem \ref{Pi_n genus and stratum} showed $g=n-1$, and we deduce that $\sigma_n$ has precisely $2n=2(n-1)+2$ fixed points, hence $\Pi_n$ is hyperelliptic.
\end{proof}

\begin{remark}\label{rem:halfishyperelliptic}
    When $n$ is even, we have $\Pi_n ^h$ is hyperelliptic since there is a degree 2 translation covering $p:\Pi_n\to\Pi_n ^h$ from a hyperelliptic surface $\Pi_n$. See Figure~\ref{Pi(6,h)} for an example when $n=6.$
\end{remark}

%%%%%%%%%%%%%%%%%%%%%%%%%%%%%%%%%%%%%%%%%%%%%%%%%%%%%%%%%%%%%%%%%%%%%%%%%%%%%%%%%%%%%%%%%%%%%%%%%%%%%%%%%%%%%%%%%%%%%

\section{Determining the lattice property}

In this section, we prove that the unfoldings $\tilde{P}_n$ are not lattice surfaces unless $n=4$. In particular, $\tilde{P}_4$ is a square-tiled surface, which is a lattice surface because it is a branched cover of the torus with one branched point. We prove $\tilde P_n$ is not a lattice surface unless $n=4$ by showing there are two cylinders with irrational ratio of their moduli. As we will see shortly, it suffices to consider $\Pi_n$.
\begin{prop}\label{IrrationalCylinder}
   Let $n\in\mathbb Z$, $n\ge3$. There is a horizontal cylinder with irrational modulus in $\Pi_n$ for all odd $n$ and a horizontal cylinder with irrational modulus in $\Pi_n ^h$ for all even $n$ other than $4$.
\end{prop}
Assuming this, we can prove the following.
\begin{thm}\label{thm:notlattice}
    The surfaces $\tilde{P}_n$, $\Pi_n$, and $\Pi_n ^h$ are not lattice surfaces unless $n=4$.
\end{thm}
\begin{proof}
    Let $n\in \mathbb Z$ be odd and $n\ge3.$ By Theorem \ref{Thm:TranslationCovering} and Proposition \ref{CoveringMap}, $\SL{\tilde{P}_n}$ is commensurable with $\SL{\Pi_n}$. Hence, it suffices to show $\SL{\Pi_n}$ is not a lattice in $\SL{2,\R}$. Using Corollary \ref{CorVeech}, we only need to find two cylinders in $\Pi_n$ such that the ratio of their moduli are not commensurable. Notice that the stem in $\Pi_n$ is a horizontal cylinder with height and circumference one, and the existence of another horizontal cylinder with irrational modulus is provided by Proposition \ref{IrrationalCylinder}. Therefore, the surface $\Pi_n$ does not satisfy Theorem \ref{VeechDichotomy} (Veech dichotomy), so it is not a lattice surface.

    When $n$ is even and different than 4, an argument similar to the previous paragraph works where we use $\Pi_n ^h$ instead of $\Pi_n$.
\end{proof}

Combining the above with results of the previous sections allows us to prove Theorem \ref{thm:UnfoldingsResult} and Theorem \ref{thm:UnfoldingsVeech}

\begin{proof}(Of Theorem \ref{thm:UnfoldingsResult} and Theorem \ref{thm:UnfoldingsVeech})
   We collect the relevant results: Lemma \ref{lem: unfoldings tilde Pn} provides us with the stratum data of $\tilde P_n$. Proposition \ref{CoveringMap} shows $\tilde P_n$ is never primitive. Theorem \ref{thm:notlattice} shows that $\tilde P_n$ is a lattice surface if and only if $n=4$. Clearly $\tilde P_4$ is a square-tiled surface.
\end{proof}

\begin{figure}
    \centering
    \includegraphics[scale=.6]{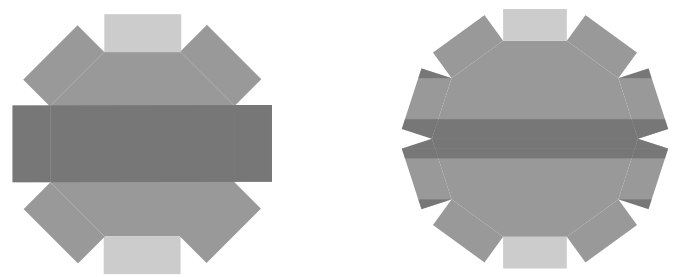}
    \caption{On the left are the two horizontal cylinders for $\Pi_{8} ^h$ and on the right are the cylinders for $\Pi_{10} ^h.$ Note that the outermost vertices are not cone points.} \label{Cyl4k+2}
\end{figure}

Section \ref{EvenCase} and Section \ref{OddCase} aim to verify the existence of the cylinder from Proposition \ref{IrrationalCylinder}, splitting into cases whether $n$ is even or odd. To determine irrationality, we use the following lemma whose proof is given in Appendix \ref{Appendix}. For convenience, we shall compute the reciprocal of the modulus in the proof of Proposition \ref{IrrationalCylinder}. 
\begin{lem}\label{IrrationalityLemma}
    Let $n$ be a positive integer. 
    \begin{enumerate}
        \item $\tan\prt{\frac{\pi}{n}}$ is irrational for $n\in\Z_{\geq 3}\setminus\lset{4}$.
        \item $\lbrkt{\Q\prt{\tan\prt{\frac{\pi}{n}}}:\Q}>2$ for $n\in\Z_{\geq 5}\setminus\lset{6,8,12}$.
    \end{enumerate}
\end{lem}

\subsection{The case for even $n$ other than 4}\label{EvenCase} Our strategy is to construct the cylinder universally. Since the choice of the cylinder depends on the geometric properties of the polygons and edge identifications, we shall consider the cases of $n\in4\Z_{\geq2}$ and $n\in4\Z_++2$.
\begin{proof}[Proof of Proposition \ref{IrrationalCylinder} for even $n$ other than $4$]
    We consider the following two cases.
    \begin{enumerate}
        \item \textbf{Case 1: $n\in4\Z_{\geq2}$.} We choose the horizontal cylinder that contains the middle horizontal strip of the base and the middle horizontal strip of the center with two attached petals. See the left polygon in Figure~\ref{Cyl4k+2}.
        
        The height of this cylinder is the length of a petal, which is $1$. As such, we need to compute the horizontal circumference of this cylinder. This amounts to recalling that \emph{half} of the width of the regular $n$-gon is $\frac{1}{2\tan (\frac{\pi}{n})}$. Since each petal has width 1/2, we obtain a total circumference of $1+\frac{1}{\tan\prt{\frac{\pi}{n}}}$. Therefore,  the reciprocal of its modulus is $1+\frac{1}{\tan\prt{\frac{\pi}{n}}}$, an irrational number by the first item of Lemma \ref{IrrationalityLemma}. 

    \begin{figure}
    \centering
    \includegraphics[scale=.4]{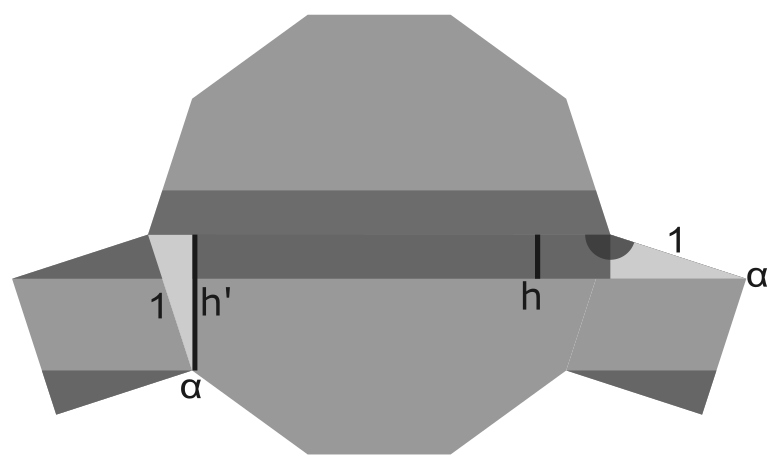}
    \caption{The triangles needed in Case 2 of Proposition \ref{IrrationalCylinder} even} \label{Triangle4k+2}
    \end{figure}
        
        \item \textbf{Case 2: $n\in4\Z_++2$.}
        We note that for $n=4k+2$, there is a horizontal line that cuts the center $n$-gon in half and we choose the horizontal cylinder just below this line.  We also do a cut and paste of the half petals just above the horizontal line that cuts the center $n$-gon in half and move them just below the horizontal line. See Figure~\ref{Cyl4k+2}.

        To prove that this cylinder is well defined, it suffices to show that $h<h'$ in general (See Figure~\ref{Cyl4k+2}) where these are heights on two triangles. The first triangle has a side of length  
        $h$ and is formed by taking the right vertex of the center $n$-gon, that we denote $v$, the vertex of the petal emanating one unit to the right of $v$ downwards, and the vertex that lies at the intersections of the vertical ray below $v$ with the horizontal ray to the right of the second vertex. The second triangle has a side length $h'$  and has vertices the  left most vertex of the center $n$-gon $v'$, and the next vertex (going counterclockwise) $v''$, along with the intersection of the horizontal ray going right from $v'$ with the vertical ray going up from $v''$. Of course, $h$ is just the height of the horizontal cylinder below the horizontal line that cuts the center $n$-gon in half, but once we show that $h<h'$, then we have this cylinder is defined.

        To compute $h$, we will show, by using the triangle in Figure~\ref{Triangle4k+2}, $h=\sin(\alpha)$. This triangle has hypotenuse 1. We compute $\alpha$ as follows: the interior angle of a regular $n$-gon is $\frac{(n-2)\pi}{n}$. As such, \emph{half} of this along with the addition of the angle $\pi/2$ from the stem gives the shaded circle arc in Figure~\ref{Triangle4k+2} a total angle of $\frac{(n-1)\pi}{n}$. Since this and $\alpha$ are adjacent angles to the non-parallel sides of the trapezoid in Figure~\ref{Triangle4k+2}, and since such angles add up to $\pi$, we conclude $\alpha=\frac{\pi}{n}$.

        To compute $h'$, we consider the triangle it lies on (See Figure~\ref{Triangle4k+2}). We note that one angle of the triangle is $\pi/2$ and the \emph{half} angle about the left-most vertex of the center $n$-gon is $\frac{(n-2)\pi}{2n}$. Since the sum of the angles is $\pi$, we see that the remaining angle is $\alpha=\frac{\pi}{n}$. In particular, $h' = \cos(\frac{\pi}{n})$.

        Since $n\ge6$ is even, and cosine is smaller than sine for positive angles less than $\pi/4$, then we see $h=\sin(\frac{\pi}{n})<\cos(\frac{\pi}{n})=h'$. This proves that this cylinder exists.

         As such, we can now compute the dimension of this cylinder from Figure~\ref{Cyl4k+2}. The height  was shown to be $\sin\prt{\frac{\pi}{n}}$ and, computing the circumference, we see it is $\frac{2}{\sin\prt{\frac{\pi}{n}}}+2\cos\prt{\frac{\pi}{n}}$. Therefore, the reciprocal modulus of this cylinder is $$\frac{\frac{2}{\sin\prt{\frac{\pi}{n}}}+2\cos\prt{\frac{\pi}{n}}}{\sin\prt{\frac{\pi}{n}}}=2\prt{1+\cot\prt{\frac{\pi}{n}}+\cot^2\prt{\frac{\pi}{n}}}.$$ When $n=6$, the reciprocal modulus is $8+2\sqrt{3}$, which is irrational. For $n\in4\Z_{\geq2}+2$, if the reciprocal modulus is rational, then $\tan\prt{\frac{\pi}{n}}=\frac{1}{\cot\prt{\frac{\pi}{n}}}$ is an algebraic number of degree $\leq2$, which contradicts the second item of Lemma \ref{IrrationalityLemma}.
    \end{enumerate}
\end{proof}

\subsection{The case for odd $n\ge3$}\label{OddCase}

    \begin{wrapfigure}{l}{0.5\textwidth}
    \includegraphics[scale=.4]{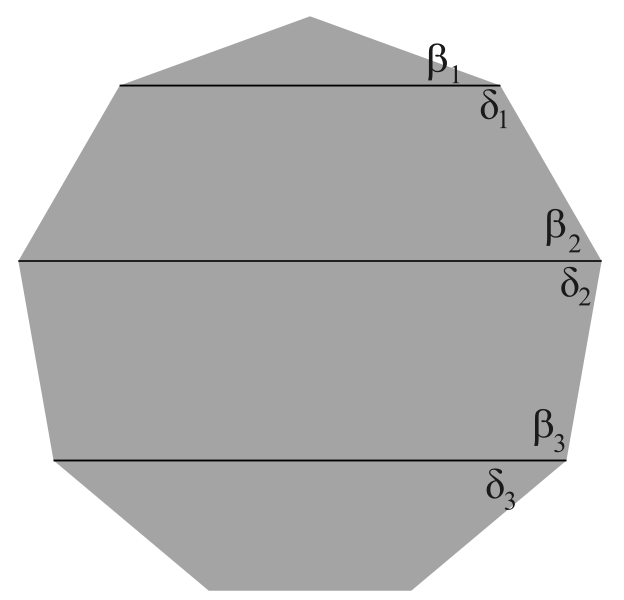}
    \caption{The layers along with the angles $\beta_{i,9}$ and $\delta_{i,9}$ on $\Pi_9.$} \label{betas}
    \end{wrapfigure}
    We will need to label and work with certain angles in the course of constructing a second cylinder with irrational modulus when $n$ is odd. To this end, we consider the horizontal line between vertices of the center $n$-gon that are on opposite sides of the vertical. This will divide the center $n$-gon into \emph{layers}. These lines divide the interior angle of the center $n$-gon into two and we define the angle above the horizontal of each layer to be $\beta$ and below the horizontal to be $\delta$.

    More formally, for each vertex of the center $n$-gon below the vertex at the ``top", there is another vertex along the same horizontal. This divides the center $n$-gon into $\frac{n-1}{2}$ layers. Calling the first horizontal line from the ``top" the bottom of the first layer, this horizontal splits the interior angle of the center $n$-gon into an angle above the horizontal, that we denote by $\beta_{1,n}$ and below that we denote by $\delta_{1,n}$. Continuing in this fashion to the next layer below, we obtain a sequence of angles $\lset{\beta_{j,n}}_j$ and $\lset{\delta_{j,n}}_j$. The angles are symmetric along the vertical edge. See Figure~\ref{betas}.

We first compute these angles:
\begin{lem}\label{AnglesOfIterations}
    For any odd $n>3$, the general term formulas of the angles are \begin{align*}
    \beta_{j,n}&=\frac{\prt{2j-1}\pi}{n}\ \text{and}\ \delta_{j,n}=\frac{\prt{n-2j-1}\pi}{n}
    \end{align*} for $j\in\lset{1,\dots, 2k}$ when $n=4k+3$ or $j\in\lset{1,\dots,2k-1}$ when $n=4k+1$. When $n=3$, there is only 1 layer. In this case and $\beta_{1,3}=\pi/3$ and $\delta_{1,3}=0$.
\end{lem}
\begin{proof}
    The interior angle of a regular $n$-gon with degree $\frac{\prt{n-2}\pi}{n}$. Since the layers are chosen to divide the interior angle of the trapezoid, we naturally have 
    \begin{equation}\label{eqn: intangle}
        \frac{\prt{n-2}\pi}{n}=\beta_{j,n}+\delta_{j,n}.
    \end{equation} For $j>1$, each layer forms a symmetric trapezoid and so since the adjacent angles $\delta_{j-1,n}$ and $\beta_{j,n}$ to the non-parallel sides must add up to $\pi$, we have  
    \begin{equation}\label{eqn: trapezoid}
    \beta_{j,n}=\pi-\delta_{j-1,n}.
    \end{equation}

    We repeatedly utilize Equation ~\ref{eqn: intangle} and Equation ~\ref{eqn: trapezoid}. We demonstrate the first few steps below.
    
    The first layer forms an isosceles triangle at the top (see Figure~\ref{betas}), which yields $$\beta_{1,n}=\frac{\pi-\frac{\prt{n-2}\pi}{n}}{2}=\frac{\pi}{n}.$$ As such, we can compute $$\delta_{1,n} = \frac{\prt{n-2}\pi}{n}-\beta_{1,n}=\frac{\prt{n-2}\pi}{n}-\frac{\pi}{n} = \frac{(n-3)\pi}{n}.$$

    Using Equation ~\ref{eqn: trapezoid}, we see that $$\beta_{2,n}=\pi - \frac{(n-1)\pi}{n}=\frac{\pi}{n}.$$
    
    Continuing in this fashion, we obtain \begin{align*}
        \beta_{j,n}&=\pi-\frac{\prt{n-2}\pi}{n}+\beta_{j-1,n}=\prt{j-1}\prt{\pi-\frac{\prt{n-2}\pi}{n}}+\beta_{1,n}=\prt{j-1}\frac{2\pi}{n}+\frac{\pi}{n}=\frac{\prt{2j-1}\pi}{n},
    \end{align*}and \begin{align*}
        \delta_{j,n}=\frac{\prt{n-2}\pi}{n}-\frac{\prt{2j-1}\pi}{n}=\frac{\prt{n-2j-1}\pi}{n}.
    \end{align*}
\end{proof}

    \begin{figure}
    \centering
    \includegraphics[scale=.4]{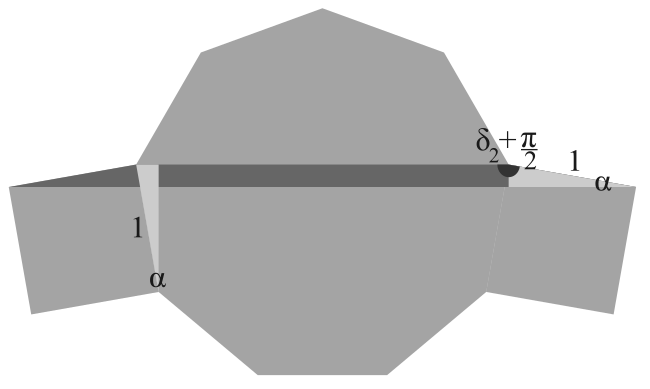}
    \caption{The triangles needed in Case 1 of Proposition \ref{IrrationalCylinder} odd} \label{OddTriangle4l+1}
    \end{figure}
We now continue the proof of Proposition \ref{IrrationalCylinder}:
\begin{proof}[Proof of Proposition \ref{IrrationalCylinder} for odd $n\geq3$] We consider the following two cases.
\begin{enumerate}
    \item \textbf{Case 1: $n\in4\Z_++1$.} Informally, the cylinder we choose is the first petal that goes from pointing ``upwards" to ``downwards".
    
    More precisely, along the vertices of the center $n$-gon, $n=4k+1$, we have $2k$ horizontal layers and choose the horizontal cylinder that has a horizontal saddle connection contained on the bottom of the $k$th layer. This saddle connection will lie on the top of the horizontal cylinder. Consider Figure ~\ref{OddTriangle4l+1} for the case $n=9$ as a reference.

     To prove that this cylinder is well defined, we consider two triangles that are similar to Case 2 of $n=4k+2$ even. The first triangle is formed by taking the right vertex of the center $n$-gon, that we denote $v$, the vertex of the petal emanating one unit to the right of $v$, and the horizontal line that backtracks to  $v$. We also consider the triangle formed by the left-most vertex of the center $n$-gon, the next vertex (going counterclockwise) on the center $n$-gon, and then backtracking vertically. Each triangle shares a common angle $\alpha$ and if $\cos(\alpha)>\sin(\alpha)$, then the cylinder is well defined.

    By Lemma \ref{AnglesOfIterations}, we have
    $$\delta_{k,n}=\frac{\prt{n-2k-1}\pi}{n}=\frac{\prt{4k+1-2k-1}\pi}{4k+1}=\frac{\pi}{2}-\frac{\pi}{2n}.$$
    As such, the angle below the horizontal that makes the $k$th layer is $\delta_{k,n}+\pi/2$. Since $\delta_{k,n}+\pi/2$ and $\alpha$ are adjacent angles to the non-parallel sides of the trapezoid in Figure~\ref{Triangle4k+2}, and since such angles add up to $\pi$, we conclude 
    $$\alpha=\pi-(\delta_{k,n}+\frac{\pi}{2})=\frac{\pi}{2}-\delta_{k,n}=\frac{\pi}{2n}.$$

    Since $n>3$, we have $\cos(\alpha)>\sin(\alpha)$ and so the cylinder is well defined. Since $\delta_{k,n}<\frac{\pi}{2}-\frac{\pi}{n}$, the circumference of the cylinder is $2\prt{\cos\prt{\frac{\pi}{2n}}+\frac{\cos\prt{\frac{\pi}{2n}}}{\sin\prt{\frac{\pi}{n}}}}$. Therefore, the reciprocal modulus is $2\prt{\cot\prt{\frac{\pi}{2n}}+\frac{1}{2\sin^2\prt{\frac{\pi}{2n}}}}$. If the reciprocal modulus is rational, then $\tan\prt{\frac{\pi}{2n}}$ is  of degree at most $2$ over $\Q$. By \ref{tanvalue}, we have, for odd $n$, $$\phi\prt{4n}=\phi\prt{\lcm\prt{2n,4}}=2\lbrkt{\Q\prt{\tan\prt{\frac{\pi}{2n}}}:\Q}\leq4,$$ where $\phi$ is Euler's totient function. Thus, there are nine cases where the Euler totient function is less than or equal to 4. That is, when $4n\in\lset{1,2,3,4,5,6,8,10,12}$, but $4n\geq 20$ for $n\in4\Z_++1$, a contradiction.
    %%%%%%%%%%%%%%%%%%%%%%%%%%%%%%%%%%%%%%%%%%%%%%%%%%%%%%%%%%%%%%%%%%%%%%%%%%%%%%%%%%%%%%%
    \item \textbf{Case 2: $n\in4\Z_{\geq0}+3$.} 
    Informally, the cylinder we choose is the first petal pointing ``upwards"  such that the subsequent petal points ``downwards". This is slightly different from our previous choices because $n=3$ has no petals on the center that point ``downwards".
    
    More precisely,  we separate the center $n=4k+3$-gon into $2k+1$ horizontal layers. We choose the horizontal cylinder that has a horizontal saddle connection contained on the bottom of the $k+1$th layer. 

    We prove this cylinder is well defined by considering two triangles and computing their relative heights. The first triangle has vertices given by the right-most vertex $v$, the vertex of the petal emanating one unit to the right of $v$ upwards,  and the vertex that lies at the intersection of the vertical ray up from $v$ with the horizontal ray left of the second vertex. The second triangle is chosen so that one vertex is the left-most point of the center $n$-gon, the next vertex (going clockwise) on the center $n$-gon, and then backtracking vertically down.

    As before, each triangle shares a common angle $\alpha$ and if $\cos(\alpha)>\sin(\alpha)$, then the cylinder is well defined.  By Lemma \ref{AnglesOfIterations}, we have 
    $$ \beta_{k+1,n}=\frac{\prt{2k+1}\pi}{n}.$$ 
    By taking into account the angle from the stem of $\pi/2$ and utilizing that adjacent angles to the non-parallel sides of a trapezoid add up to $\pi$, we deduce 
    and consider 
        $$\alpha=\frac{\pi}{2}-\beta_{k+1,n}=\frac{\pi}{2n}.$$
     As such, the cylinder is well defined, and we can compute the circumference to be
     $$2\prt{\cos\prt{\beta_{k+1,n}}+\frac{\cos\prt{\beta_{k+1,n}-\prt{\frac{\pi}{2}-\frac{\pi}{n}}}}{\sin\prt{\frac{\pi}{n}}}}=2\prt{\cos{\prt{\frac{\pi}{2n}}}+\frac{\cos\prt{\frac{\pi}{2n}}}{\sin\prt{\frac{\pi}{n}}}}.$$ Therefore, the reciprocal modulus is $2\prt{\cot{\prt{\frac{\pi}{2n}}}+\frac{1}{2\sin^2\prt{\frac{\pi}{2n}}}}$. If the reciprocal modulus is rational, then the same argument as in the case $n\in4\Z_++1$ would give $4n\in\lset{1,2,3,4,5,6,8,10,12}$. Since $n\in4\Z_{\geq0}+3$, the only possible $n$ for the reciprocal modulus to be rational is when $n=3$. Yet, a direct computation shows that the corresponding reciprocal modulus is $4+2\sqrt{3}$, a contradiction.
\end{enumerate}
\end{proof}

%%%%%%%%%%%%%%%%%%%%%%%%%%%%%%%%%%%%%%%%%%%%%%%%%%%%%%%%%%%%%%%%%%%%%%%%%%%%%%%%%
\section{Primitivity and consequences}
In this section we show that $\Pi_n$ is primitive for odd $n\ge3$ and $\Pi_n ^h$ is primitive for even $n\ge5$. The proof of primitivity is inspired by a similar argument from Hooper \cite{MR3071661}. Primitivity and hyperellipticity will allow us to use the classification of Apisa \cite{MR3769676} of hyperelliptic components of strata to deduce information about orbit closure and asymptotics on the number of saddle connections on these surfaces.

We recall the notion of \emph{balanced covering}.  A covering $p:(X,\omega)\to(Y,\eta)$ of translation surfaces is a balanced covering if the image of every zero of $\omega$ is a zero of $\eta$. We require the following result of Hooper \cite{MR3071661}.

\begin{prop}
     Let $(X,\omega)$ be a translation surface which does not cover a torus. Let $p : X \to Y$ be the unique covering of a primitive translation surface $(Y, 
     \eta)$ guaranteed by Theorem \ref{thm:Moller}. If the collection of affine automorphisms $\text{Aff}(X, \omega)$ acts transitively on the zeros of $\omega$, then $p$ is a balanced covering.
\end{prop}

\begin{prop}\label{prop:balanced}
    Let $n\ge3$. 

    If $n$ is odd, then $p:\Pi_n\to Y $ is balanced where $Y$ is the primitive surface guaranteed by Theorem \ref{thm:Moller}.

    If $n$ is even and $n\ne4$, then $p:\Pi_n ^h\to Y $ is balanced where $Y$ is the primitive surface guaranteed by Theorem \ref{thm:Moller}.
\end{prop}

\begin{proof}
    First note that for $n\ne4$, Theorem \ref{thm:notlattice} shows that neither $\Pi_n$ nor $\Pi_n^h$  cover a torus since they are not lattice surfaces. The proposition is now an immediate consequence of hyperellipticity (Proposition \ref{prop:hyperelliptic} and Remark \ref{rem:halfishyperelliptic}). Indeed, since the surfaces we consider are hyperelliptic, then $-\text{Id}\in\text{Aff}(\Pi_n)$ and this involution swaps the two zeros. In particular, when $n$ is odd, then $p$ is balanced. If $n$ is even, then $\Pi_n ^h$ is in the minimal stratum and so $p$ is balanced.
\end{proof}

\begin{thm} \label{thm:primitive}
The surface $\Pi_n$ is primitive for odd $n\ge3$ and $\Pi_n ^h$ is primitive for even $n\ge5$.
\end{thm}

\begin{proof}
First note that for $n\ne4$, Theorem \ref{thm:notlattice} shows that neither $\Pi_n$ nor $\Pi_n^h$ cover a torus since they are not lattice surfaces. In particular, there is a unique translation covering $p:\Pi_n\to Y$ for odd $n$ and $p:\Pi_n ^h\to Y$ for even $n\ge5$. By Proposition \ref{prop:balanced}, these coverings are balanced.

Suppose for the sake of contradiction that the degree $d$ of $p$ is at least 2.

By Proposition \ref{prop:balanced}, $p$ sends the zeros of the covering surface to the zeros of the base. hence, it sends saddle connections to saddle connections. Thus, for a convex polygon $P$ in the covering surface whose boundary edges are all saddle connections, we have $p(P)$ also has this property in the base surface. Then for any such $P$, the collection of polygons in $p^{-1}(p(P))$ is a collection of $d$ isometric polygons with disjoint interiors that are bounded by saddle connections and differ only by translation.

Now, let $n\ge3$ be odd or $n=4k+2$ for a natural number $k\ge1$. In this case, there are exactly two horizontal and vertical unit length saddle connections on $\Pi_n$ (when $n$ is odd) or $\Pi_n ^h$ (when $n=4k+2$ is even).

Consider $P$ given by the stem of the covering surface which is a unit square with horizontal and vertical sides that lies on the bottom horizontal edge of the center $n$-gon. Note these are the shortest of all the saddle connections of the surface. Then $p^{-1}(p(P))$ is a collection of $d$ isometric squares with disjoint interiors that are bounded by saddle connections and differ only by translation. Since there are exactly two horizontal unit length saddle connections, then any square in $p^{-1}(p(P))$ must use these saddle connections. But, above one of the two horizontal unit length saddle connections is $P$. Above the other horizontal unit length saddle connection, there is no unit length vertical saddle connection and so no other unit square with horizontal and vertical unit length sides. This is a contradiction and proves primitivity of  $\Pi_n$ (when $n$ is odd) or $\Pi_n ^h$ (when $n=4k+2$ is even).

Now we prove primitivity for $\Pi_n ^h$ in the remaining case that $n=4k+4$ for a natural number $k\ge1$. In this case, there are exactly three horizontal and vertical unit-length saddle connections. Since $\Pi_n^h$ is in the minimal stratum, then the base surface $Y$ is also in the minimal stratum. Since $k\ge1$, then, by Theorem \ref{thm:notlattice}, $Y$ is not a torus. 

Each saddle connection of $Y$ lifts to exactly $d$ saddle connections on $\Pi_n ^h$. Let $U$ denote the number of horizontal unit-length saddle connections. Then since we know all horizontal unit length saddle connections on $\Pi_n^h$, we have $dU=3$. Since we are assuming, for contradiction, $d\ne1$, we obtain that $U=1$ and $d=3$. That is, $p$ is a degree 3 translation covering. In particular, the stem $P$ of $\Pi_n ^h$ which is unit square with horizontal and vertical sides on the bottom horizontal edge of the center $n$-gon is a convex polygon whose boundary edges are all saddle connections. As such, $p^{-1}(p(P))$ is a collection of $3$ isometric squares with disjoint interiors that are bounded by saddle connections and differ only by translation. One such square is given by the stem, and another such square is given by the petal on the vertical edge of the center $n$-gon in $\Pi_n ^h$. The remaining square $S$ is bounded above and below by horizontal unit-length saddle connections. One of those horizontal unit length saddle connections is the bottom side of the stem, and another one of those is the top and bottom edge of the petal on the vertical edge of the center $n$-gon in $\Pi_n^h$. Since there are exactly three horizontal unit length saddle connections, then the remaining one $\gamma$ must be used as the bottom of $S$. The left and right sides of $S$, must be vertical unit length saddle connections, but the vertical saddle connection emanating from either end point of $\gamma$ is on the top of the $n$-gon. In particular, it is not unit-length, and this is a contradiction.
\end{proof}

With this we can prove Theorem \ref{thm:main}.
\begin{proof} [Proof of Theorem \ref{thm:main}]

    Suppose $n$ is even and different than $4.$ By Theorem \ref{thm:notlattice} we have $\Pi_n ^h$ is not a lattice surface and hence by the classification of Apisa \cite{MR3769676} of hyperelliptic components of strata (Theorem 5) the $\text{GL}(2,\mathbb R)$ orbit closure of $\Pi_n ^h$ must arise from some rank $r>1$ branched covering construction over a hyperelliptic minimal stratum of lower genus. On the other hand, by Theorem \ref{thm:primitive}, $\Pi_n^h$ is primitive, so that the $\text{GL}(2,\mathbb R)$ orbit closure of $\Pi_n ^h$ must coincide with $\mathcal H^{\text{hyp}}(n-2)$.

    In particular, by Eskin--Mirzakhani--Mohammadi \cite{MR3418528}, $\Pi_n ^h$ satisfies Equation \ref{eqn: EMM} for the Siegel-Veech constant associated to $\mathcal H^\text{hyp}(n-2)$. This constant was computed by  Eskin--Masur--Zorich \cite{MR2010740}. %there are exact asymptotics for the number of saddle connections in a ball. That is, there is a constant $c=c(\mathcal H(n-2))>0$ such that 
        %$$\lim_{L\to\infty}\frac{N(\Pi_n^h,L)}{\pi L^2}=\frac{c}{\text{Area}(\Pi_n^h)}.$$

        Now suppose  $n\ge3$ is odd. By Proposition \ref{IrrationalCylinder} and Theorem \ref{thm:primitive} we have $\Pi_n $ is primitive and not a lattice surface. By  the classification of Apisa \cite{MR3769676} of hyperelliptic components of strata (Theorem 6) the $\text{GL}(2,\mathbb R)$ orbit closure of $\Pi_n$ is either rank 1 or dense in $\mathcal H^{\text{hyp}}(n-2,n-2)$. Hence, if the orbit closure is not rank 1, then it must be dense. On the other hand, if the orbit closure had rank 1, then by Theorem 1.1 of Apisa \cite{MR4034916}, the orbit closure is either dense, closed, or arises from a branched covering construction and we may use Proposition \ref{IrrationalCylinder} and Theorem \ref{thm:primitive} to conclude the orbit closure must still be dense. That is, in either case, we have the orbit closure is dense and so $\overline{\text{GL}(2,\mathbb R)\cdot\Pi_n}=\mathcal H^{\text{hyp}}(n-2,n-2).$ 
        
        In particular, by Eskin--Mirzakhani--Mohammadi \cite{MR3418528}, $\Pi_n ^h$ satisfies Equation \ref{eqn: EMM} for the Siegel-Veech constant associated to $\mathcal H^\text{hyp}(n-2,n-2)$. This constant was computed by  Eskin--Masur--Zorich \cite{MR2010740}.
\end{proof}

With this we can prove Corollary \ref{cor:asymforPn}.

\begin{proof} [Proof of Corollary \ref{cor:asymforPn}]
    
    We first note that the degree $n$ cover from Proposition \ref{CoveringMap} $p:\tilde P_n\to \Pi_n$ is balanced since the zeros of $\tilde P_n$ get sent to the zeros of $\Pi_n$. As such, $N(\tilde {P_n},L)=nN(\Pi_n,L)$ for any $n\ge3$ and any $L>0$. Moreover, there are no regular pre-images of poles of $P_n$ in $\tilde P_n $ so that $N(\tilde P_n,L)=nN( P_n,L)$. We will use these equalities shortly.

    Furthermore, by Theorem \ref{thm:main}, we have
        $$\lim_{L\to\infty}\frac{1}{\pi^2L}\int_0^LN(\Pi^h _n,e^t)e^{-2t}dt=\frac{c}{\text{Area}(\Pi_n^h)}\text{ and }\lim_{L\to\infty}\frac{1}{\pi^2L}\int_0^LN(\Pi_n,e^t)e^{-2t}dt=\frac{d}{\text{Area}(\Pi_n)}.$$

    where $n\ge3$ is even in the first equality and $n\ge3$ is odd in the second.

    In particular, for $n\ge3$ and odd, we have
    \begin{align*}
    \lim_{L\to\infty}\frac{1}{\pi  L^2}\int_0^LN(P _n,e^t)e^{-2t}dt&=\lim_{L\to\infty}\frac{1}{n\pi  L^2}\int_0^LN(\tilde P_n ,e^t)e^{-2t}dt\\
    &=\lim_{L\to\infty}\frac{1}{n\pi  L^2}\int_0^LnN(\Pi_n,e^t)e^{-2t}dt\\
    &=\frac{d}{\text{Area}(\Pi_n)} =\frac{d}{\text{Area}(P_n)} .
    \end{align*}

   A similar argument works in the $n\ge3$ even case. Assume $n\ne4$. Note the degree $2n$ cover from $\tilde P_n\to\Pi_n^h$ is balanced since it sends zeros to zeros. We have,
       \begin{align*}
        \lim_{L\to\infty}\frac{1}{\pi  L^2}\int_0^LN(P _n,e^t)e^{-2t}dt&=        \lim_{L\to\infty}\frac{1}{n\pi  L^2}\int_0^LN(\tilde P _n,e^t)e^{-2t}dt\\
        &=\lim_{L\to\infty}\frac{2n}{n\pi  L^2}\int_0^LN(\Pi ^h _n,e^t)e^{-2t}dt\\
        &=\frac{2c}{\text{Area}(\Pi_n^h)}=\frac{4c}{\text{Area}(P_n)}.
        \end{align*}
    \end{proof}

This finishes the proofs of the main results.

\appendix
\section{The Irrationality Lemma}\label{Appendix} In this appendix, we give the proof of Lemma \ref{IrrationalityLemma}:
\begin{proof}
\begin{enumerate}
    \item Notice that, by the identity $\tan\prt{\alpha+\beta}=\frac{\tan\prt{\alpha}+\tan\prt{\beta}}{1-\tan\prt{\alpha}\tan\prt{\beta}}$, we have $\tan\prt{k\theta}$ is rational for $k\in\Z_+$ if $\tan\prt{\theta}$ is rational. Writing $n=2^am$, since $n\in\Z_{\geq 3}\setminus\lset{4}$, then we have
    \begin{itemize}
        \item for $m=1$, $n$ is at least $8$ and $8\mid n$, and we write $n=8c_1$ with $c_1\in\Z_+$
        \item for $m>1, m\neq 2$, $m$ has an odd prime divisor $p$, and we write $n=pc_2$ with $c_2\in\Z_+$.
    \end{itemize}
        So, if we can prove the irrationality of $\tan\prt{\frac{\pi}{8}}$ and $\tan\prt{\frac{\pi}{p}}$ with odd primes $p$, then $\frac{\pi}{8}=c_1\frac{\pi}{n}$ and $\frac{\pi}{p}=c_2\frac{\pi}{n}$ provide the irrationality of $\tan\prt{\frac{\pi}{n}}$.\\
Recall that $\tan\prt{\frac{\pi}{3}}=\sqrt{3}, \tan\prt{\frac{\pi}{5}}=\sqrt{5-2\sqrt{5}}$, and $\tan\prt{\frac{\pi}{8}}=\sqrt{2}-1$, which are all irrational. So, it suffices to consider primes larger than $5$. Towards contradiction, for $p>5$, suppose that $\tan\prt{\frac{\pi}{p}}$ is rational. Since $\sin\prt{\frac{\pi}{p}}=\frac{e^{\frac{i\pi}{p}}-e^{-\frac{i\pi}{p}}}{2i}$ and the set of algebraic numbers forms a field, $\sin\prt{\frac{\pi}{p}}$ is an algebraic number. Further, $\sin\prt{\frac{\pi}{p}}=\frac{\tan\prt{\frac{\pi}{p}}}{\sqrt{1+\tan^2\prt{\frac{\pi}{p}}}}$ and $\tan\prt{\frac{\pi}{p}}\in\Q$ force $\sin\prt{\frac{\pi}{p}}$ to have degree at most $2$ over $\Q$. Recall that $\cos\prt{\frac{\pi}{p}}=\frac{1}{\sqrt{1+\tan^2\prt{\frac{\pi}{p}}}}$, so $\sin\prt{\frac{\pi}{p}},\cos\prt{\frac{\pi}{p}}\in\Q\prt{\sqrt{1+\tan^2\prt{\frac{\pi}{p}}}}$ and $\sqrt{1+\tan^2\prt{\frac{\pi}{p}}}=\frac{1}{\cos\prt{\frac{\pi}{p}}}$ imply that $\Q\prt{\sin\prt{\frac{\pi}{p}},\cos\prt{\frac{\pi}{p}}}=\Q\prt{\sqrt{1+\tan^2\prt{\frac{\pi}{p}}}}$. Thus, we have \begin{align*}
    \lbrkt{\Q\prt{\sin\prt{\frac{\pi}{p}},\cos\prt{\frac{\pi}{p}}}:\Q}=\lbrkt{\Q\prt{\sqrt{1+\tan^2\prt{\frac{\pi}{p}}}}:\Q}\leq 2.
\end{align*} Let $\zeta_{2p}$ be the $2p$-th root of unity over $\Q$, then $\labs{\Q\prt{\zeta_{2p}}:\Q}=p-1$. Yet, since $\zeta_{2p}=e^\frac{i\pi}{p}=\cos\prt{\frac{\pi}{p}}+i\sin\prt{\frac{\pi}{p}}$, we have \begin{align*}
    4<p-1=\lbrkt{\Q\prt{\zeta_{2p}}:\Q}\leq\lbrkt{\Q\prt{\sin\prt{\frac{\pi}{p}},\cos\prt{\frac{\pi}{p}},i}:\Q}\leq 4,
\end{align*} whence the result. 
\item If $p\mid n$, then $\Q\prt{\tan{\prt{\frac{\pi}{p}}}}\subset\Q\prt{\tan\prt{\frac{\pi}{n}}}$, which implies that $\lbrkt{\Q\prt{\tan\prt{\frac{\pi}{p}}}:\Q}\leq\lbrkt{\Q\prt{\tan\prt{\frac{\pi}{n}}}:\Q}$. For $n\geq 7$, the only possibilities for primes $p\geq7$ that do not cover are those $n$ having only $2,3,5$ as the prime factors. Yet, $\tan\prt{\frac{\pi}{5}}=\sqrt{5-2\sqrt{5}}$ is an algebraic number of degree $\geq 4$ over $\Q$, so the only $n$ we have to consider are of the form $2^a3^b$. Let $\zeta_n$ be the $n$-th root of unity. Since $\tan\prt{\frac{\pi}{n}}=\frac{\zeta_n-1}{i\prt{\zeta_n+1}}$, we have $\Q\prt{\tan\prt{\frac{\pi}{n}},i}=\Q\prt{\zeta_n,i}=\Q\prt{\zeta_n,\zeta_4}=\Q\prt{\zeta_{\lcm\prt{n,4}}}$. So, the tower rule provides \begin{align}\label{tanvalue}
    \lbrkt{\Q\prt{\tan\prt{\frac{\pi}{n}}}:\Q}&=\frac{ \lbrkt{\Q\prt{\tan\prt{\frac{\pi}{n}},i}:\Q}}{\lbrkt{\Q\prt{\tan\prt{\frac{\pi}{n}},i}:\Q\prt{\tan\prt{\frac{\pi}{n}}}}}\nonumber\\&=\frac{\lbrkt{\Q\prt{\zeta_{\lcm\prt{n,4}}}:\Q}}{2}\\
    &=\frac{\phi\prt{\lcm\prt{n,4}}}{2}\nonumber,
\end{align} where $\phi$ is the Euler's totient function. Writing $n=2^a3^b$, we only have to rule out the cases for $n=6,8,12$.\\
Now, towards contradiction, suppose $\labs{\Q\prt{\tan\prt{\frac{\pi}{n}}}:\Q}\leq2$ for $n\in \Z_{\geq7}\setminus\lset{8,12}$. Back to $p\geq7$, since $\tan\prt{\theta}=\frac{\sin\prt{\theta}}{\cos\prt{\theta}}$, we have $\Q\prt{\sin\prt{\frac{\pi}{p}},\cos\prt{\frac{\pi}{p}}}=\Q\prt{\tan\prt{\frac{\pi}{p}},\cos\prt{\frac{\pi}{p}}}$. Over $\Q\prt{\tan\prt{\frac{\pi}{p}}}$, $\cos\prt{\theta}$ satisfies $\prt{1+\tan^2\prt{\theta}}\cos^2\prt{\theta}-1=0$. Hence, we have 
    \begin{align*}
        \lbrkt{\Q\prt{\sin\prt{\frac{\pi}{p}},\cos\prt{\frac{\pi}{p}}}:\Q}&=\lbrkt{\Q\prt{\tan\prt{\frac{\pi}{p}},\cos\prt{\frac{\pi}{p}}}:\Q}\\
        &=\lbrkt{\Q\prt{\tan\prt{\frac{\pi}{p}},\cos\prt{\frac{\pi}{p}}}:\Q\prt{\tan\prt{\frac{\pi}{p}}}}\lbrkt{\Q\prt{\tan\prt{\frac{\pi}{p}}}:\Q}\\
        &\leq 2\cdot 2\\
        &=4.
    \end{align*} Let $\zeta_{2p}$ be the $2p$-th root of unity. As in $(1)$, We have \begin{align*}
        \lbrkt{\Q\prt{\zeta_{2p}}:\Q}\leq\lbrkt{\Q\prt{\sin\prt{\frac{\pi}{p}},\cos\prt{\frac{\pi}{p}},i}:\Q}\leq 2\cdot 4=8.
    \end{align*} 
    If $p=7$, the identities $\sin\prt{\frac{\pi}{7}}=\frac{\zeta_{14}-\zeta_{14}^{-1}}{2i}$ and $\cos\prt{\frac{\pi}{7}}=\frac{\zeta_{14}+\zeta_{14}^{-1}}{2}$ imply that
    \begin{align*}
12=\lbrkt{\Q\prt{\zeta_{28}}:\Q}=\lbrkt{\Q\prt{\zeta_{14},i}:\Q}=\lbrkt{\Q\prt{\sin\prt{\frac{\pi}{7}},\cos\prt{\frac{\pi}{7}},i}:\Q}\leq 8.
    \end{align*}
If $p>7$, we have $10\leq p-1=\lbrkt{\Q\prt{\zeta_{2p}}:\Q}\leq 8$. Either case achieves a contradiction.
\end{enumerate}
\end{proof}

\end{document}